\documentclass[12pt]{article}
\usepackage[russian, english]{babel}
\usepackage {amsmath}
\usepackage {amsthm}
\usepackage {amssymb}
\usepackage{wrapfig}
\usepackage{amsfonts}
\usepackage{textcomp}
\usepackage{graphicx}
\usepackage{float}
\usepackage{caption}
\usepackage{algorithm}
\usepackage{algpseudocode}
\usepackage{xcolor}
\usepackage{hyperref}
\usepackage{mathenv}
\usepackage{booktabs}
\usepackage{indentfirst}
\usepackage{enumerate}
\usepackage{bbold}
\usepackage{dsfont}

\newtheorem{Thm}{Theorem}[section]
\newtheorem{Lemm}{Lemma}[section]

\theoremstyle{definition}

\let\le=\leqslant

\let\leq=\leqslant
\let\geq=\geqslant

\begin{document}
\baselineskip 12pt

\begin{center}
\textbf{\Large On the Asymptotics of the Connectivity Probability \\ of Random Bipartite Graphs} \\

\vspace{1.5cc}
{ B.B. Chinyaev$^{1}$}\\
\vspace{0.3 cm}

{\small $^{1}$M.V. Lomonosov Moscow State University, bchinyaev.msu@gmail.com }
\end{center}
\vspace{1.0cc}

\begin{abstract}
In this paper, we analyze the exact asymptotic behavior of the connectivity probability in a random binomial bipartite graph $G(n,m,p)$ under various regimes of the edge probability $p=p(n)$. To determine this probability, a method based on the analysis of inhomogeneous random walks is proposed.

\vspace{0.95cc}
\parbox{24cc}{{\it \textbf{Keywords:} random bipartite graph, connectivity probability, \\ exact asymptotics, random walks}
}
\end{abstract}

\section{Introduction} \label{intro}

The random bipartite graph model \( G(n,m,p) \) is frequently studied in the theory of random graphs, as well as in various applied problems in combinatorics, probability theory, and computer science. In this model, the vertex set is partitioned into two parts, \( V_1 = \{1, \dotsc, n\} \) and \( V_2 = \{1, \dotsc, m\} \), and each edge connecting a vertex in \( V_1 \) with a vertex in \( V_2 \) is included independently with probability \( p = p(n,m) \). This model serves as a natural generalization of the classical Erdős–Rényi graph \( G(n,p) \) to the bipartite setting.

For the classical $G(n,p)$ model, several asymptotic results concerning the connectivity probability have been established (see, e.g., \cite{stepanov1970probability}). These results provide a detailed description of the behavior of the connectivity probability as $p(n)$ tends to zero at various rates. In our recent work \cite{chinyaev2024er_eng}, a novel approach based on the theory of inhomogeneous random walks was proposed to investigate these asymptotics, thereby rederiving known formulas and unveiling new results. Building on this approach, in \cite{chinyaev2025method} we developed an efficient method for sampling connected $G(n,p)$ and $G(n,m)$ random graphs.

The aim of this paper is to analyze the asymptotic behavior of the connectivity probability $P_{n,m}(p)$ for the random bipartite graph $G(n,m,p)$ by employing a methodology similar to that used in our previous work \cite{chinyaev2024er_eng} on Erdős–Rényi graphs. During the investigation, we derived a novel and convenient nonasymptotic representation of the connectivity probability in terms of a sequential graph exploration process. This representation, which is applicable regardless of the relations among the parameters $n$, $m$, and $p(n,m)$, is discussed in detail in Section~\ref{main_lem}. The analysis of the asymptotics of the derived expression is reduced to the study of a process which, although not a inhomogeneous random walk as in the previous case, is analyzed using methods developed for inhomogeneous walks.

In the present version of the paper, we prove a key lemma that reduces the analysis of the connectivity probability of the bipartite random graph $G(n,m,p)$ to the study of a process constructed from two inhomogeneous random walks. This representation enables a detailed asymptotic analysis, and we formulate a theorem describing the exact asymptotics behavior of the connectivity probability under various regimes of $p = p(n,m)$ as $n, m \to \infty$. The full proof of the theorem will be included in a future version, while the present results have independent interest and provide a foundation for further work.

\section{Preliminaries} \label{prelim}

In our previous work \cite{chinyaev2024er_eng}, we established a connection between the connectivity probability of the graph $G(n,p)$ and the probability that a certain inhomogeneous Poisson random walk remains nonnegative, as formalized in the following lemma.
\begin{Lemm}[\cite{chinyaev2024er_eng}]\label{lem_Gnp}
Let $G(n,p)$ be an Erdős–Rényi graph. Then the connectivity probability is given by
$$
P_n (p) = \left(1-(1-p)^{n}\right)^{n-1} {\bf P}(S_k \geq 0, \ 0< k < n| S_{n} = -1 ), 
$$
where  $S_k = \sum_{i=1}^{k} X_i$, and $X_i$ are independent random variables such that $X_i + 1\sim Poiss\left(\lambda_i\right)$, with 
 $$\lambda_i = \frac{np}{1-(1-p)^{n}}(1-p)^{(i-1)}.$$
\end{Lemm}
\noindent The proof of this lemma was based on the analysis of the graph exploration process. In this paper, we consider a similar process for the bipartite graph case.

\subsection{Exploration process in the Bipartite Graph} 

To study the connectivity of the random bipartite graph $G(n, m, p)$, we use exploration process (see in \cite{zakharov2023asymptotic}) analogous to that used for $G(n,p)$ (see  \cite{karp1990transitive}, \cite{nachmias2010critical}). Initially, all vertices are inactive except for one designated vertex from the left set, which is activated.  At each step, a vertex is selected from the set of activated but not yet processed vertices. If this set contains any vertices from the right part, the earliest activated one is processed; otherwise, the earliest activated vertex from the left part is selected.

When the $i$th activated vertex from the left part is processed, all of its inactive neighbors in the right part are activated. The number of such neighbors is denoted by $R_i$. We denote by $S^R_k$ the total number of vertices in the right part that have been activated after processing the first $k$ vertices from the left part. Similarly, when the $j$th activated vertex from the right part is processed, all of its inactive neighbors in the left part are activated. We denote their number by $L_j$, while $S^L_k$ denotes the total number of left vertices activated after processing the first $k$ vertices from the right set. Then,

\begin{equation}
S^R_{i+1} = S^R_i + R_i , \ i< n, \ \ \ S^L_{j+1} = S^L_j + L_j, \ j< m.
\end{equation}
Since the edges in \(G(n,m,p)\) occur independently with probability \(p\), the following holds:
\begin{equation}
	\mathbf{P}\left(R_i = k  \ | S^R_{i}=l\right) = 
	\left\{\begin{array}{ll}
            \binom{m-l}{k} \  p^k (1-p)^{m-l}, &  i \le S^L_l, \\
            0, &  i > S^L_l.\\
        \end{array} \right.
    \label{BReq}
\end{equation}

\begin{equation}
	\mathbf{P}\left(L_j = k  \ | S^L_{j}=l\right) = 
	\left\{\begin{array}{ll}
            \binom{n-1-l}{k} \ p^k (1-p)^{m-l}, &  j \le S^R_{l}, \\
            0, &  j > S^R_{l}.\\
        \end{array} \right.
\end{equation}
It is convenient to introduce \(S^*_k\), the number of active non-processed left vertices remaining after processing the \(S^R_k\) right vertices. That is, 
\begin{equation}
S^*_{k} = S^L_{S^R_k} - k, \ k\le n.
\end{equation}
The graph is connected if and only if the exploration process continues until every vertex is activated; that is, \(S^*_k>0\) for \(k=1,2,\dotsc,n-1\) and \(S^R_n = m\). Consider the trajectories of the exploration process, denoted by \(\{R_i, L_j\}\). Let
\[
\mathbf{r}=(r_1,\dotsc,r_n) \quad \text{and} \quad \mathbf{l}=(l_1,\dotsc,l_m)
\]
be the observed sequences of \(\{R_i\}\) and \(\{L_j\}\), respectively.
In order for \(G\) to be connected, the sequences \(\mathbf{r}\) and \(\mathbf{l}\) must satisfy
\begin{equation}
T_{n,m} =\left\{\mathbf{r} ,\mathbf{l}  :\sum _{i=1}^{n} r_{i} =m, \ \sum _{j=1}^{\sum _{i=1}^{k}r_i} l_{j} \geqslant k, \ k< n, \ \sum _{j=1}^{n} l_{j} =n-1\right\},
\end{equation}
Thus, the connectivity probability can be expressed as follows:
\begin{equation}\label{Pcon}
 \begin{aligned}
 P_{n,m} (p)=\sum\limits _{(\mathbf{r} ,\mathbf{l} )\in T_{n,m}}  \left( \prod_{i=1}^{n} C_{m-r_{1} -\dotsc -r_{i-1}}^{r_{i}} p^{r_{i}} (1-p)^{r_{i+1} +\dotsc +r_{n}}\right)  \times 
 \\
\times \left( \prod _{j=1}^{m} C_{n-1-l_{1} -\dotsc -l_{j-1}}^{l_{j}} p^{l_{j}} (1-p)^{l_{j+1} +\dotsc +l_{m}}\right) . 
\end{aligned}
\end{equation}

\subsection{Connectivity Probability and Random Walks}\label{main_lem}

The following lemma provides a reformulation of the connectivity probability into a more analytically tractable form.

\begin{Lemm}[Connectivity Probability of $G(n,m,p)$]\label{connect}
Let $G(n,m,p)$ be a random binomial bipartite graph. Let \(X_i\) and \(Y_j\) be independent random variables satisfying
\begin{align} \label{ABS_def1}
\begin{split}
X_i\sim Poiss\left(\alpha_i\right), \quad
\alpha_i = \frac{mp}{1-(1-p)^{n}}(1-p)^{i-1}, \quad i\leq n, \\
Y_j\sim Poiss\left(\beta_j\right), \quad 
\beta_j = \frac{np}{1-(1-p)^{m}}(1-p)^{j-1}, \quad j\leq m. 
\end{split}
\end{align}
Define the processes
\begin{equation} \label{ABS_def3}
    A_k = \sum_{i=1}^{k} X_i, \ \ B_k = \sum_{j=1}^{k} Y_j, \ \ S_{k} = B_{A_k} - k.
\end{equation}
Then, the connectivity probability of $G(n,m,p)$ can be expressed as
\begin{align}\label{main_lem_res}
\begin{split}
P_{n,m}(p) = \left(1-(1-p)^{n}\right)^{m} \left(1-(1-p)^{m}\right)^{n-1} \\
\times \
{\bf P}(S_k \geq 0, 0< k < n \ |\ A_{n} = m, B_{m} = n-1 ).
\end{split}
\end{align}
\end{Lemm}

\begin{proof}
We now transform expression (\ref{Pcon}). The terms under the summation can be written as follows
\begin{align}
\prod _{i=1}^{n} C_{m-r_{1} -\dotsc -r_{i-1}}^{r_{i}} p^{r_{i}} (1-p)^{r_{i+1} +\dotsc +r_{n}} =p^{m} m!\prod _{i=1}^{n}\left(\frac{(1-p)^{(i-1)r_{i}}}{r_{i} !}\right), \label{mainlem_prod1}\\
\prod _{j=1}^{m} C_{n-1-l_{1} -\dotsc -l_{j-1}}^{l_{j}} p^{l_{j}} (1-p)^{l_{j+1} +\dotsc +l_{m}} =p^{n-1}( n-1) !\prod _{j=1}^{m}\left(\frac{(1-p)^{(j-1)l_{j}}}{l_{j} !}\right) . \label{mainlem_prod2}
\end{align}
For $\displaystyle q >0$, the product in (\ref{mainlem_prod1}) can be rewritten as
\begin{equation*}
\exp\left( q\sum _{i=1}^{n} (1-p)^{i-1}\right) q^{-m}\prod _{i=1}^{n}\left(\exp\left( -q(1-p)^{i-1}\right)\frac{q^{r_{i}} (1-p)^{(i-1)r_{i}}}{r_{i} !}\right) .
\end{equation*}
Let
\begin{equation*}
q=\frac{mp}{1-(1-p)^{n}} =m\left(\sum _{i=1}^{n} (1-p)^{i-1}\right)^{-1} .
\end{equation*}
Then the expression (\ref{mainlem_prod1}) becomes
\begin{equation*}
m!\exp( m)\left(\frac{1-(1-p)^{n}}{m}\right)^{m}\prod _{i=1}^{n}\exp( -\alpha _{i})\frac{\alpha _{i}^{r_{i}}}{r_{i} !} , \quad\alpha _{i} =\frac{mp}{1-(1-p)^{n}} (1-p)^{i-1}.
\end{equation*}
Applying similar transformations to the product in~\eqref{mainlem_prod2}, we obtain the following expression:
\begin{equation*}
( n-1) !\exp( n)\left(\frac{1-(1-p)^{m}}{n}\right)^{n-1}\prod _{j=1}^{m}\exp( -\beta _{j})\frac{\beta _{j}^{l_{j}}}{l_{j} !} ,\ \ \beta _{j} =\frac{np}{1-(1-p)^{m}} (1-p)^{j-1} .
\end{equation*}
Therefore, the expression (\ref{Pcon}) is transformed into the form
\begin{equation}
\begin{gathered}
 P_{n,m} (p)=
\left( 1-(1-p)^{n}\right)^{m}\left( 1-(1-p)^{m}\right)^{n-1} \frac{m!\exp( m)}{m^{m}}\frac{( n-1) !\exp( n)}{n^{n-1}}\times \\
\times \sum\limits _{(\mathbf{r} ,\mathbf{l} )\in T_{n,m}}\left(\prod _{i=1}^{n}\exp( -\alpha _{i})\frac{\alpha _{i}^{r_{i}}}{r_{i} !}\right)\left(\prod _{j=1}^{m}\exp( -\beta _{j})\frac{\beta _{j}^{l_{j}}}{l_{j} !}\right). 
\end{gathered}
\end{equation}
Note that if $X_{i} \sim Poiss(\alpha _{i} )$ and $Y_{j} \sim Poiss(\beta _{j} )$, then 
\begin{equation*}
\exp( -\alpha _{i})\frac{\alpha _{i}^{r_{i}}}{r_{i} !} ={\bf P}(X_{i} =r_{i} ), \ \ \exp( -\beta _{j})\frac{\beta _{j}^{l_{j}}}{l_{j} !} ={\bf P}(Y_{j} =l_{j} ).
\end{equation*}
Combining the expressions above, we obtain
\begin{equation}
\begin{gathered}
 P_{n,m} (p) = \left( 1-(1-p)^{n}\right)^{m}\left( 1-(1-p)^{m}\right)^{n-1}{\bf P}(A_{n} = m, B_{m} = n-1 )^{-1} \\
\times {\bf P}(S_k \geq 0, 0< k < n, \ A_{n} = m, B_{m} = n-1 ).
\end{gathered}
\end{equation}
Hence, the desired form (\ref{main_lem_res}) immediately follows.
\end{proof}
The lemma above serves as a bipartite analogue of Lemma~\ref{lem_Gnp}. It provides a tool for analyzing the connectivity probability of the graph \( G(n,m,p) \) under various regimes of the parameter \( p \). To this end, we aim to compute the probability that the random process \( S_k \) remains nonnegative throughout, conditioned on returning to \( -1 \) at the final step. We begin by calculating the expected values of \( A_k \), \( B_k \), and \( S_k \).
\begin{equation}\label{mu_eta}
\mu _{k} =\mathds{E} A_{k} =\sum _{j=1}^{k} \alpha _{j}  =m \frac{1-( 1-p)^{k}}{1-( 1-p)^{n}}, \quad
\eta _{k}  = \mathds{E} B_{k} =\sum _{i=1}^{k} \beta_i =n \frac{1-( 1-p)^{k}}{1-( 1-p)^{m}} , 
\end{equation}
\begin{gather*}
\mathds{E} S_{k}  = \sum _{i=0}^{\infty }\mathbf{P}( A_{k} =i)  \eta _{i}  =\sum _{i=0}^{\infty } e^{-\mu _{k}}\frac{\mu _{k}^{i}}{i!} \ n \frac{1-( 1-p)^{i}}{1-( 1-p)^{m}} = \\
=\frac{n}{1-( 1-p)^{m}}\left( 1-e^{-\mu _{k} p}\right) .
\end{gather*}
Note that the increments of the process \( \{S_k\} \) initially have a positive expected value, but this expectation decreases with \( k \) and eventually becomes negative.
\begin{figure}[H]
    \centering
    \includegraphics[width=0.8\linewidth, height= 0.4\linewidth]{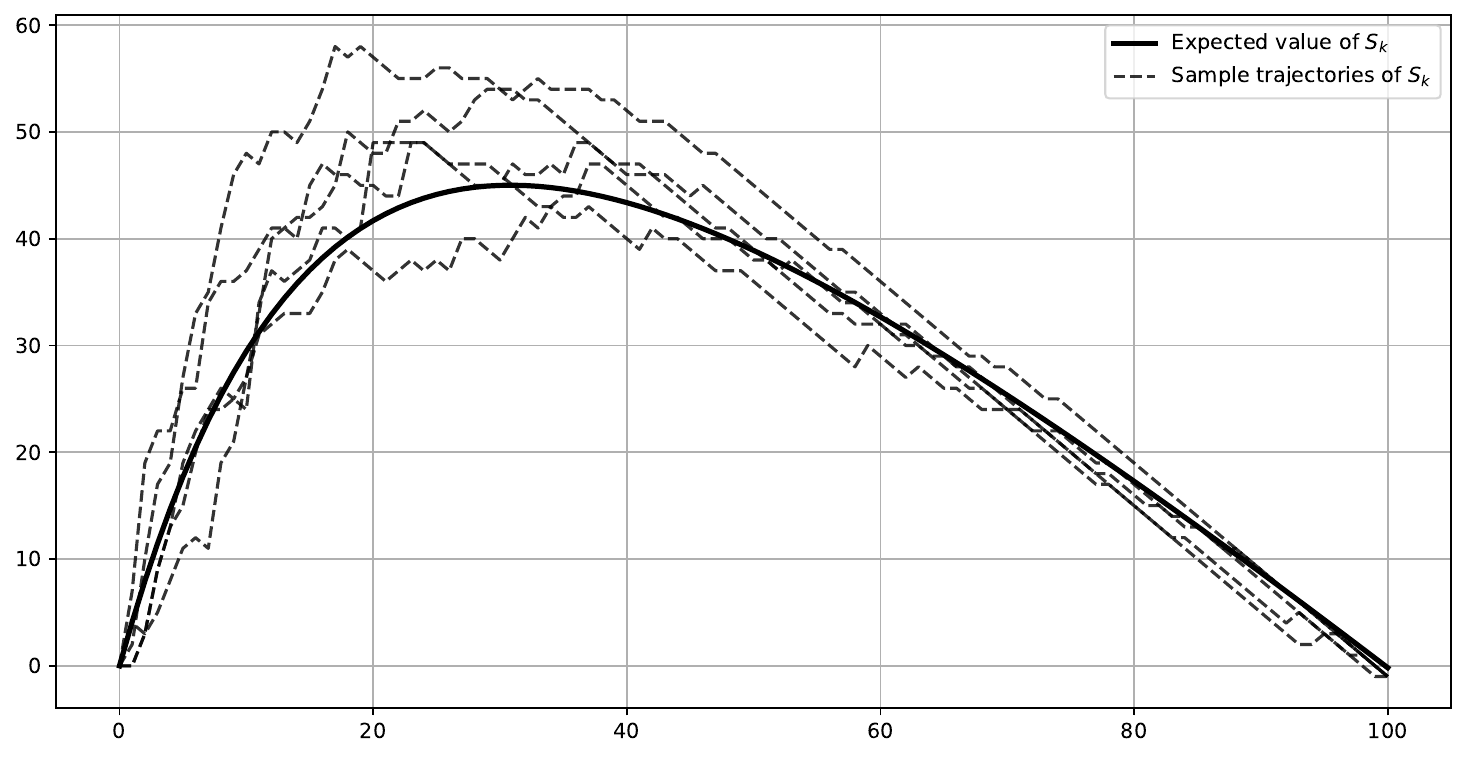}  
    \caption{Plot of the expected value and sample realizations of $S_k$.}
\end{figure}\label{ex1}
\noindent It is convenient to consider the inhomogeneous recovery process
\begin{equation*}
V_{k}^{A}  = \sum _{i=1}^{\infty } I( A_{i} \leqslant k).
\end{equation*}
Then the nonnegativity condition for \( S_k \) can be equivalently expressed as
\begin{equation*}
\left\{S_{k} =B_{A_k} -k\geqslant 0,\ k \le n\ \right\} \Leftrightarrow  \left\{B_{k} -V_{k}^{A} \geqslant  0,\ k \le n \right\}.
\end{equation*}
\begin{figure}[H]
    \centering
    \includegraphics[width=0.8\linewidth, height= 0.4\linewidth]{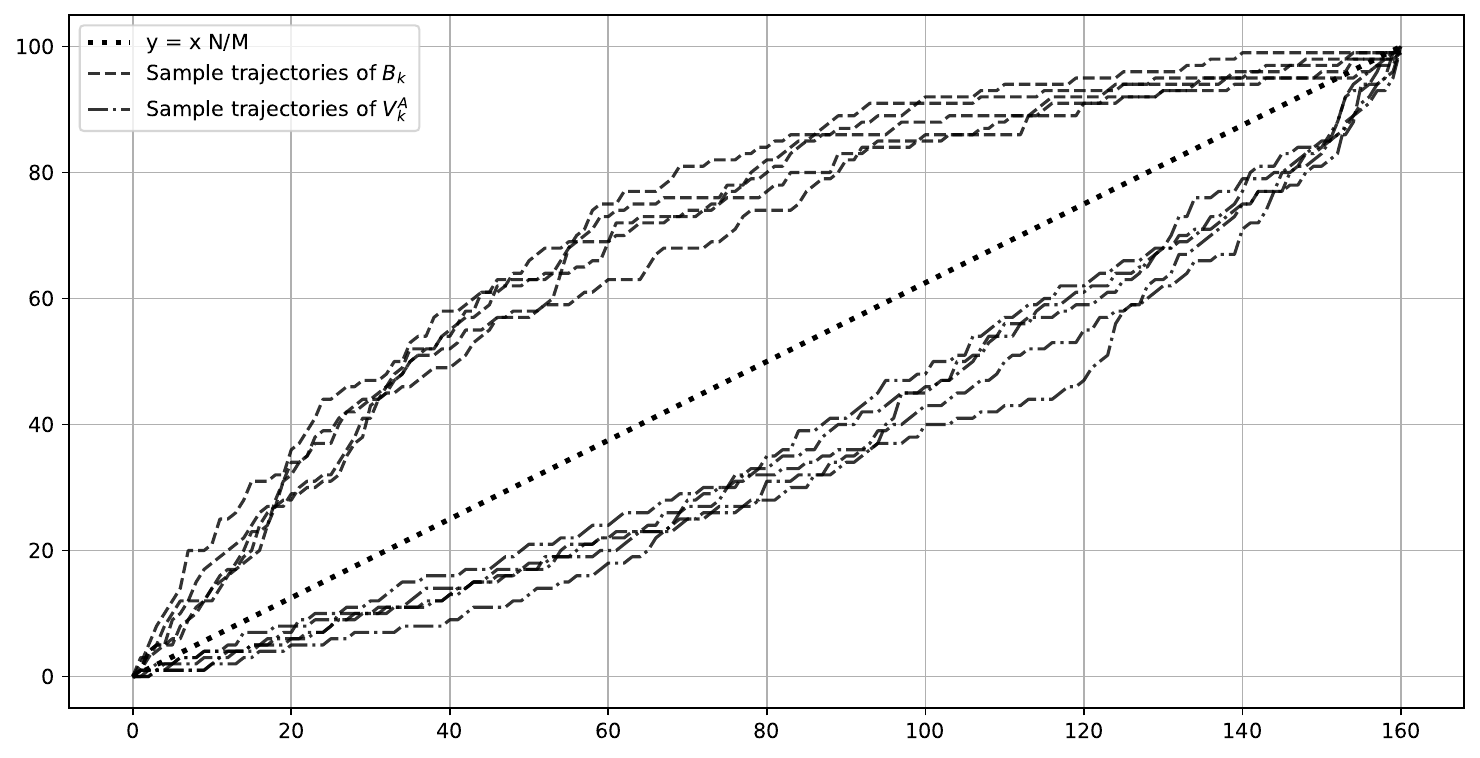}  
    \caption{Plot of sample realizations of $B_k$ and $V_{k}^{A}.$}
\end{figure}\label{ex2}

\noindent As observed from the simulations, the trajectories of \( B_k \) typically lie above those of \( V_k^A \). Note that neither of them even crosses the line $y(k) = k n/m$. We make use of this observation by formally demonstrating the 
rarity of this event in the proof of the main result.
\newpage
We show that the conditional probability in (\ref{main_lem_res}) is monotonic with respect to~$p$.
\begin{Lemm}[Monotonicity] 
Let $0<p_1<p_2<1$ and let the processes $\{A^i_k,B^i_k,S^i_k\}$ be defined by (\ref{ABS_def1}) and (\ref{ABS_def3}) for the corresponding $p=p_i$. Then
\begin{align}
\begin{split}
&{\bf P}(S^1_k \geq 0, 0< k < n| A^1_{n} = m, \ B^1_{m} = n-1 )\le 
\\
\le \ &{\bf P}(S^2_k \geq 0, 0< k < n| A^2_{n} = m, \ B^2_{m} = n-1)\leq 1. \label{compare_lem_res}
\end{split}
\end{align}
\end{Lemm}
\begin{proof}
We prove this claim using a coupling argument on a single probability space. Consider Poisson processes $N_t^A$ and $N_t^B$ with intensity $1$. Then, denoting by $\mu_k^i$ and $\eta_k^i$ the quantities defined in (\ref{mu_eta}) for $p=p_i$, for $i\in \{1,2\}$, we have
\begin{equation*}
\left(A^i_k, \ k\le n\right)\stackrel{d}{=} (N_{\mu_1^i}^A,N_{\mu_2^i}^A,\dotsc, N_{\mu_n^i}^A), \ \
\left(B^i_k, \ k\le m\right)\stackrel{d}{=} (N_{\eta_1^i}^B,N_{\eta_2^i}^B,\dotsc, N_{\eta_m^i}^B).
\end{equation*}
Note that $N_{\mu_n^1}^A=N_{\mu_n^2}^A=N_{m}^A$, $N_{\eta_m^1}^B=N_{\eta_m^2}^B=N_{n}^B$, and moreover, $N_{\mu_k^2}^A$ and $N_{\eta_k^2}^B$ dominate $N_{\mu_k^1}^A$ and $N_{\eta_k^1}^B$, respectively, due to the monotonicity
\begin{equation*}
k(m/n)\leq\mu_k^1\leqslant\mu_k^2, \ k<n, \quad k(n/m)\leq\eta_k^1\leqslant\eta_k^2, \ k<m.
\end{equation*}
Therefore, in this probability space
\begin{equation*}
\left\{B_{A_k^1 } \geq k, k < n, A^1_{n} = m, B^1_{m} = n-1   \right\} \subseteq \left\{B_{A_k^2 } \geq k, k < n, A^2_{n} = m,  B^2_{m} = n-1   \right\}.
\end{equation*}
Hence, the desired claim follows.
\end{proof}

\subsection{Formulation of the Main Theorem}\label{main_theorem}

Using the representation obtained in Lemma \ref{connect}, we describe the asymptotic behavior of the connectivity probability of the graph $G(n,m,p)$ under various regimes for the parameter $p = p(n,m)$ as $n,m \to \infty$. The following theorem summarizes the obtained results.
\begin{Thm}[On the Connectivity Probability of $G(n,m,p)$]\label{main res}
Let $G(n,m,p)$ be a random binomial bipartite graph with edge probability $p = c_n/(n+m)$ and $m \sim a n$, as $n \to \infty$, where $a > 0$.
Let $P_{n,m}(p)$ be the probability that the graph $G(n,m,p)$ is connected.
\begin{enumerate}[1)]
\item 
Assume that $c_n \to +\infty$, as $n \to \infty$. Then
\begin{equation}\label{Gnm_res1}
P_{n,m}(p) \sim  \left(1-(1-p)^{n}\right)^{m} \left(1-(1-p)^{m}\right)^{n},\quad 
n\to \infty.
\end{equation}
\item
Assume that $c_n \to c \in (0,+\infty)$, as $n \to \infty$. Then
\begin{equation}
\label{Gnm_res2}
\begin{gathered}
P_{n,m}(p) \sim  (1-\alpha_n\beta_m)\left(1-(1-p)^{n}\right)^{m} \left(1-(1-p)^{m}\right)^{n}, \\
\alpha_n\sim \frac{mc}{n+m} \frac{e^{-c n/(n+m)}}{1-e^{-c n/(n+m)}}, \ \beta_m\sim \frac{nc}{n+m} \frac{e^{-c m/(n+m)}}{1-e^{-c m/(n+m)}}, \ \  n \to \infty . 
\end{gathered}
\end{equation}

\item  
Assume that $c_n=o(1)$ and that $c_n n^{1/2}/\ln n \to+\infty$, as $n\to\infty$. Then 
\begin{equation}\label{Gnm_res3}
P_{n,m}(p) \sim \frac{c_n}{2}\left(1-(1-p)^{n}\right)^{m} \left(1-(1-p)^{m}\right)^{n} ,\quad 
n\to \infty.
\end{equation}
\item 
Assume that $c_n=o(1/n)$, as $n \to \infty$. Then
\begin{equation}\label{Gnm_res4}
P_{n,m}(p) \sim \frac{1}{n} \left(1-(1-p)^{n}\right)^{m} \left(1-(1-p)^{m}\right)^{n-1} \sim n^{m-1}m^{n-1} p^{n+m-1},\ 
n\to \infty.
\end{equation}
\end{enumerate}
\end{Thm}

\section{Conclusion} \label{conclusion}

In this work, we proposed a new approach for analyzing the connectivity probability of random bipartite graphs $G(n,m,p)$, based on a sequential vertex exploration process. The main result is a nonasymptotic representation of the connectivity probability in terms of the conditional probability that the process $S_k = B_{A_k} - k$, built from two inhomogeneous random walks, remains nonnegative. This representation provides a convenient tool for deriving exact asymptotics under various regimes of the parameter~$p$.

We formulated a theorem describing the asymptotic behavior of the connectivity probability as $n, m \to \infty$ under several scaling regimes of $p = p(n,m)$. Although the complete proof will be included in a future version, the results presented — particularly the proven lemma — have independent value and can serve as a basis for further work. 
Possible directions for future research include:
\begin{itemize}
  \item A more detailed analysis of the structure of bipartite graphs under different relations between $n$ and $m$.
  \item The development of algorithms for efficiently generating connected bipartite graphs.
  \item The extension of the proposed method to random hypergraphs, multipartite graphs, or models with edge constraints.
\end{itemize}
We hope that the obtained results will prove useful both in the theoretical analysis of random graphs and in practical problems of modeling and analyzing network structures.


\bibliography{lit2_eng}

\end{document}